\documentclass[11pt,reqno]{article}
\usepackage{authblk}
\usepackage{amsmath,amsfonts,amssymb,latexsym,epsfig}
\usepackage{a4,color,epsf}
\usepackage{graphicx}
\usepackage{alltt}
\usepackage{enumerate}

\newcommand{\R}{\mathbb{R}}
\newcommand{\C}{\mathbb{C}}
\newcommand{\N}{\mathbb{N}}

\usepackage{soul}

\newtheorem{theorem}{Theorem}

\newtheorem{definition}{Definition}
\newtheorem{lemma}{Lemma}

\newtheorem{proposition}{Proposition}
\newtheorem{remark}{Remark}
\newenvironment{proof}[1][Proof]{\textbf{#1.} }
     {\noindent \rule{0.5em}{0.5em}}

\title{Global time-renormalization of the gravitational $N$-body problem}

\author[$\dagger$]{M. Anto\~nana}
\author[$\star$]{P. Chartier}
\author[$\dagger$]{J. Makazaga}
\author[$\dagger$]{A. Murua}

\affil[$\dagger$]{University of the Basque Country (UPV/EHU),  Donostia-San Sebasti\'an, Spain.}
\affil[$\star$]{Univ Rennes, INRIA-IRMAR, Rennes, France. }

\begin{document}
\maketitle
\begin{abstract}
This work considers the {\em gravitational} $N$-body problem and introduces global time-renormalization {\em functions} that allow the efficient numerical integration with fixed time-steps. First, a lower bound of the radius of convergence of the solution to the original equations is derived, which suggests an appropriate time-renormalization. In the new fictitious time $\tau$, it is then proved that any solution exists for all $\tau \in \mathbb{R}$, and that it is uniquely extended as a holomorphic function to a strip of fixed width. As a by-product, a global power series representation of the solutions of the $N$-body problem is obtained. Noteworthy, our global time-renormalizations remain valid in the limit when one of the masses vanishes. Finally, numerical experiments show the efficiency of the new time-renormalization functions for some $N$-body problems with close encounters. 
\end{abstract}

\section{Introduction}
In this paper, we are concerned with the solution of the so-called $N$-body problem, which considers $N$ masses $m_i, \, i = 1,\ldots,N$ in a three-dimensional space, moving under the influence of gravitational forces. Each mass $m_i$ will be described by its {\em position} $q_i \in \R^3$ and its {\em velocity} $v_i \in \R^3$ . Following Newton, the product of  mass and acceleration $m_i \frac{d^2 q_i}{dt^2}$ is equal to the sum of the forces applied on this mass and the gravitational force exerted on mass $m_i$ by a single mass $m_j$ is given by
$$
F_{ij} =\frac{G m_i m_j}{\|q_j-q_i\|^3} (q_j-q_i), 
$$
where $G$ is the gravitational constant and $\|q_j-q_i\|$ is the distance between $q_i$ and $q_j$ (in the euclidean norm).  The $N$-body equations, defined for positions $q$ outside the set 
\begin{equation*}
\Delta = \{(q_1,\ldots,q_N) \in \mathbb{R}^{3N}\ : \ q_i=q_j \mbox{ for some } i \neq j \},
\end{equation*}
are then obtained by summing over all masses 
\begin{equation} 
\label{eq:Nbody2}
m_i \frac{d^2 q_i}{dt^2}= \sum_{j \neq i} \frac{G m_i m_j}{\|q_j-q_i\|^3} (q_j-q_i).
\end{equation}

In physics and in celestial mechanics in particular,  the $N$-body problem is of great importance for the prediction of the dynamics of objects interacting with each other. At first, the motivation for considering this problem was the desire to understand how the sun, the moon and other planets  and stars are moving. Later on, other questions like the {\em stability} of the Solar System became subjects of much inquiry in astronomy. For instance, Jacques Laskar of the {\em Bureau des Longitudes} in Paris published in 1989 the results of numerical integrations of the Solar System over $200$ million years \cite{laskar} and showed that the earth's orbit is chaotic: a mere error of  $15$ meters in its position today makes it impossible to predict its motion over $100$ million years.

In fact, the problem of finding the general solution of the $N$-body problem was already considered very important and challenging in the late $19^{th}$ century when King Oscar II of Sweden  established a prize for anyone who could represent the solution of the $N$-body problem in the form of a convergent series (in a variable that is some known function of time),  under the assumption that no two points ever collide. Although no one could manage to solve the original problem, the prize was awarded to Henri Poincar\'e for his seminal contribution \cite{poincare}. In  1909, the original problem was finally solved  for $N = 3$ by Karl Fritiof Sundman \cite{sundman},  a Finnish mathematician, who was awarded a prize by the French Academy of Science for this work,
and  in the 1990s, Qiu-Dong Wang obtained a generalization of Sundman's solution for the $N \geq 3$-case~\cite{wang}. Closely related results where obtained by L. K. Babadzanjanz in~\cite{babadzanjanz} (see also~\cite{babadzanjanz2}).
As an intermediate step to obtain their globally convergent series expansion, they considered the solution of the $N$-body problem as a function of a new independent variable $\tau$ related to the physical time $t$ by
\begin{equation}
\label{eq:s(q)}
\frac{d\tau}{d t} = s(q(t))^{-1}, \quad \tau(0)=0,
\end{equation}
with some appropriate {\em time-renormalization function} $s(q)$ depending on the positions $q = (q_1,\ldots,q_N) \in \R^{3N}$. In Sundman's work (resp. in Wang's approach for the $N>3$ case), the solution is shown to be uniquely determined as a holomorphic function of the complex variable $\tau$ in a strip along the real axis (resp. in a less simple neighborhood of the real axis). Finally, a conformal mapping $\tau \mapsto \sigma$ is used to define the solution as a holomorphic function  in the unit disk, which immediately leads to the existence of a globally convergent expansion of the solution as a power series in $\sigma$.
Despite the theoretical interest of the power series expansions considered by Sundman and Wang, they are  useless as practical methods to solve the $N$-body equations due to the extremely slow convergence of the series expansions of the solutions for the values of $t$ corresponding to values of $\sigma$ near the boundary of the unit disk. As already invoked, highly accurate numerical integration schemes are required instead. 

In the presence of close encounters,   numerical integration has to be used either with an efficient strategy to adapt the step-size or in combination with some time-renormalization. In the later case, one aims at determining a real analytic function $s(q)$ that relates the physical time $t$  with a new independent variable $\tau$ by (\ref{eq:s(q)})
that allows one to use constant time-step (in $\tau$) without degrading the accuracy of the computed trajectories. 
Actually, the time-renormalizations considered by Sundman and  Wang respectively, for the 3-body problem and  the $N$-body problem respectively, can be useful for that purpose, although with some limitations.

In this paper, we introduce new time-renormalization functions that arise somewhat naturally by considering estimates of the domain of existence of the holomorphic extension of maximal solutions of the $N$-body problem
\begin{align} 
\label{eq:Nbodyq}
     \frac{d q_i}{d t} &= v_i,  & i=1,\ldots,N,\\
     \label{eq:Nbodyv}
  \frac{d v_i}{d t} &=   \sum_{\substack{j=1 \\ j\neq i}}^N \frac{G m_j}{\|q_i - q_j\|^3}(q_j-q_i), & i=1,\ldots,N,
\end{align}
 to the complex domain. 
Such time-renormalization functions depend on both positions $q = (q_1,\ldots,q_N) \in \R^{3N}$ and velocities $v = (v_1,\ldots,v_N) \in \R^{3N}$, and following Wang's terminology, define a  blow-up time-transformation. That is, the maximal  solution $(q(t),v(t))$ of (\ref{eq:Nbodyq})--(\ref{eq:Nbodyv}) 
 supplemented with initial conditions 
\begin{equation}
\label{eq:icond}
q(0) = q^0 \in \mathbb{R}^{3N} , \quad v(0) = v^0 \in \mathbb{R}^{3N}
\end{equation}
is transformed by the renormalization of time 
\begin{equation}
\label{eq:theta}
\tau=\theta(t) = \int_0^t s(q(t'),v(t'))^{-1}\, dt'
\end{equation}
into $(\hat q(\tau), \hat v(\tau))$ defined on $ \mathbb{R}$.

Our goal is to obtain  {\em uniform global time-renormalization} functions in the following sense.
\begin{definition} \label{def:timereg} Let $s(q,v)$ be a real-analytic function defined on $\mathbb{R}^{3N}\backslash\Delta \times \mathbb{R}^{3N}$. We will say that $s(q,v)$ determines a {\em uniform global time-renormalization} of (\ref{eq:Nbodyq})--(\ref{eq:Nbodyv}) if there exists $\beta>0$ such that, for  arbitrary $q^0 \in \mathbb{R}^{3N}\backslash\Delta$, $v^0 \in \mathbb{R}^{3N}$, the solution $(\hat q(\tau),\hat v(\tau))$ in $\tau$ (such that $(\hat q(\theta(t)),\hat v(\theta(t)) = (q(t), v(t))$) can be extended analytically to the strip 
$$
\{\tau \in \mathbb{C}\ : \ |\mathrm{Im}(\tau)| \leq \beta\}.
$$
\end{definition}
Our main contribution in this paper is to explicitly construct uniform global time-renormalization functions.

Notice that uniform global time-renormalizations are in particular blow-up time-transformations, but the converse does not hold. Specifically, the blow-up time transformation considered in~\cite{wang} does not satisfy the condition of Definition~\ref{def:timereg}.

{ It is worth mentioning that Sundman's time-renormalization differ from those considered  in~\cite{wang} and in the present work in that it allows, in combination  with a change of state variables, to extend the solutions as holomorphic functions of $\tau$ beyond binary collisions. This is however restricted to 3-body problems with non-zero angular momentum.}



In practice, if $s(q,v)$ is a uniform global time-renormalization function with parameter $\beta$, then one can accurately integrate the transformed equations with a high-order integrator with a constant step-size $\Delta \tau$ chosen as a moderate fraction of $\beta$.
Although standard adaptive step-size implementations may deal with the numerical integration of $N$-body problems involving close encounters, applying a suitable time-renormalization and numerically integrating the resulting equations with constant step-size is often more efficient. 
This is particularly so for long-time numerical integration: constant step-size geometric integrators~\cite{hairer06gni} applied to appropriately time-renormalized equations is indeed expected to outperform standard adaptive step-size integrators.  In addition, time-renormalization is particularly useful for approximating periodic solutions of $N$-body problems as truncated Fourier series. Indeed, the coefficients of the Fourier series expansion of periodic solutions $q(t)$ with close encounters decay (exponentially) at a very low rate, while the exponential decay of the Fourier coefficients of $(\hat q(\tau),\hat v(\tau))$ is bounded from below by the quotient of $\beta$ with the period in $\tau$. A similar situation occurs with the Chebyshev series expansion of solutions in a prescribed time interval $[0,T]$.  In this case, the decay of the Chebyshev coefficients  is determined by the size of the complex domain of analyticity (the size of the largest ellipse with foci in $0$ and $T$).

We conclude this introductory section with the outline of the article. \\

 In Section \ref{sect:main} we will state the main results: In Subsection~\ref{ss:extension}, we treat 
  the extension of the solutions of (\ref{eq:Nbodyq})--(\ref{eq:Nbodyv}) as
  holomorphic functions of the complex time;  a uniform global time-renormalization function (in the sense of Definition~\ref{def:timereg}) is exposed  in Subsection~\ref{ss:renormalization}, and a global power series representation of the solutions of (\ref{eq:Nbodyq})--(\ref{eq:Nbodyv}) (similar to those of Sundman and Wang) is presented as a by-product; the section is closed providing alternative time-renormalization functions in Subsection~\ref{ss:altfcn}. In Section \ref{sect:proofs} we give technical details of the proofs of the main results of Section~\ref{sect:main}.  Section \ref{sect:numer} is devoted to present some illustrative numerical experiments. Finally, some concluding remarks are presented in Section~\ref{sect:conclusions}.

\section{Main results} 
\label{sect:main}

In this section, we present the main contributions of the paper. 
First, a lower bound of the radius of convergence of the solution to the original
equations is derived (Theorem~\ref{th:L}). That result motivates our choice of  time-renormalization function, which is shown (see Theorem~\ref{th:s}) to be a uniform global time-renormalization function in the sense of Definition~\ref{def:timereg}.

For the sake of clarity,  the proofs of some results stated in the present section are postponed to Section~\ref{sect:proofs}. With this in mind, consider the solution $(q(t),v(t))$ of the $N$-body problem (\ref{eq:Nbodyq})--(\ref{eq:Nbodyv})
 supplemented with initial conditions (\ref{eq:icond}) such that $q^0\not \in \Delta$. Cauchy-Lipschitz theorem ensures the existence and uniqueness of a {\em maximal} solution  $q(t) \in \mathbb{R}^{3N}\backslash\Delta$, $t \in (t_a,t_b)$. 
\begin{theorem}[Painlev\'e~\cite{painleve}]
\label{th:Painleve}
If $t_a> -\infty$ (resp. $t_b < +\infty$), 
then $\min_{ij} \|q_i(t)-q_j(t)\|$ tends to 0 
as $t \uparrow t_a$ (resp. $t \downarrow t_b$). 
\end{theorem}

Given that (\ref{eq:Nbodyv}) is real analytic, it is well known that the maximal solution $q(t)$ is a real analytic function of $t$. 
For each $t^* \in (t_a,t_b)$,  the Taylor expansion of $q(t)$ at  $t=t^*$ 
$$
q(t) = q(t^*) + (t-t^*) v(t^*)  + \sum_{k \geq 2} \frac{(t-t^*)^k}{k!} q^{(k)}(t^*)
$$
is converging for all complex times $t$ such that $|t-t^*| < \rho(q^*,v^*)$,  where the {\em radius of convergence} $\rho(q^*,v^*)$ is uniquely determined by  $q(t^*)=q^* \in \R^{3N}$ and  $v(t^*)=v^* \in \R^{3N}$. In particular, $q(t)$ can be analytically extended to the complex neighbourhood of the real interval $(t_a,t_b)$
\begin{equation}
\label{eq:W}
\mathcal{W} = \left\{ t \in \mathbb{C}\ : \ |\mathrm{Im}(t)| < \rho(q(t^*),v(t^*)), \  t^* = \mathrm{Re}(t) \in (t_a,t_b) \right\}.
\end{equation}

\begin{remark}
\label{rem:W}
Provided that $t_a>-\infty$ (resp. $t_b<+\infty$),  $\rho(q(t^*),v(t^*)) < t^*-t_a$ (resp. $\rho(q(t^*),v(t^*))< t_b-t^*$), and hence $\rho(q(t^*),v(t^*))$ shrinks to 0 as $t^*  \downarrow t_a$ (resp. $t^* \uparrow  t_b$). If a close encounter occurs around $t^* \in (t_a,t_b)$, then  $\mathcal{W}$ becomes very narrow at $t^*$. 
\end{remark}


\subsection{Holomorphic extension of the solution of the $N$-body problem}
\label{ss:extension}

In the present subsection, we derive lower bounds of $\rho(q,v)$ for arbitrary $q \in \R^{3N}\backslash \Delta$ and $v \in \R^{3N}$ which define an explicit complex neighbourhood  of $(t_a,t_b)$ (included in (\ref{eq:W}))  where the solution $q(t)$ can be analytically extended. By means of Cauchy estimates, bounds of high-order derivatives of the components of $q(t)$ can be obtained straightforwardly.

We first rewrite (\ref{eq:Nbodyv}) as $dv_i/dt = g_i(q)$, $i=1,\ldots,N$, where
\begin{equation}
\label{eq:gi}
g_i(q) =   \sum_{\substack{j=1 \\ j\neq i}}^N \frac{G m_j}{((q_i - q_j)^T(q_i - q_j))^{3/2}}(q_j-q_i).
\end{equation}
Function $g_i$ is now holomorphic in the open domain of $q \in \mathbb{C}^{3N}$ obtained by excluding its singularities, i.e. values of $q \in \mathbb{C}^{3N}$ such that 
\begin{equation*}
(q_i - q_j)^T(q_i - q_j)=0
\end{equation*}
for some $(i,j)$, $1 \leq i <  j \leq N$. Note that here, denoting $q_i-q_j = (x_{ij},y_{ij},z_{ij})$, 
$(q_i - q_j)^T(q_i - q_j) = x_{ij}^2+y_{ij}^2+z_{ij}^2=0$ does not imply $q_i = q_j$.
In the sequel, the following lemma will prove of much use.
\begin{lemma}
\label{lem:tech}
Let $u$ and $U$ be two vectors of $\C^3$. The following two estimates hold true
\begin{align*}
(i) \quad & |U^T U- u^Tu|   \leq 2 \, \|u\| \, \|U-u\| + \|U-u\|^2, \\
(ii) \quad & |U^T U| \geq  |u^Tu| - 2 \, \|u\|\, \|U-u\| - \|U-u\|^2.
\end{align*}
\end{lemma}
\begin{proof}[Proof of Lemma~\ref{lem:tech}]
In order to prove $(i)$, we write $U^T U-u^T u$ as $(U-u)^T((U-u)+2u)$ and use Cauchy-Schwartz inequality to get 
$$
|(U-u)^T((U-u)+2u)| \leq \|U-u\| \, (\|U-u\| + 2 \|u\|).
$$
Point $(ii)$ follows from $|u^Tu|-|U^TU| \leq |U^TU-u^Tu|$.
\end{proof} \\ \\
Now,  given  $q^0 \in \mathbb{R}^{3N}\backslash\Delta$, consider $q \in \mathbb{C}^{3N}$ such that 
\begin{equation}
\label{eq:qcond}
 \| q_i-q_j-q^0_i+q^0_j\| < (\sqrt{2}-1) \, \|q^0_i-q^0_j\|, \ 1 \leq i < j \leq N.
\end{equation}
Lemma~\ref{lem:tech} then implies that
$(q_i - q_j)^T(q_i - q_j)>0$ for all $i,j$, so that (\ref{eq:gi})  is holomorphic
in the neighborhood of $q^0$
\begin{equation}
\label{eq:Ulambda}
\mathcal{U}_{\lambda}(q^0) = \{q \in \mathbb{C}^{3N}\ : \  \| q_i-q_j-q^0_i+q^0_j\| \leq \lambda\, \|q^0_i-q^0_j\|, \ 1 \leq i < j \leq N
\}
\end{equation}
for $\lambda = \sqrt{2}-1$.  A further application of Lemma~\ref{lem:tech} allows to bound  (\ref{eq:gi}) as in next proposition, whose proof is straightforward and thus omitted. 
\begin{proposition}
\label{prop:boundgi}
Let $\lambda \in (0,\sqrt{2}-1)$ and $q^0 \in \mathbb{R}^{3N}\backslash\Delta$. For all $q \in \mathcal{U}_{\lambda}(q^0)$, one has 
$$
\left\| g_i(q) \right\| \leq 
\eta(\lambda) \, K_i(q),
$$
where 
\begin{align}
\label{eq:Ki}
K_i(q) = \sum_{\substack{j=1 \\ j\neq i}}^N  \frac{G\, m_j}{||q_i - q_j||^2}, \qquad 
\eta(\lambda) = \frac{1+\lambda}{(1 - 2 \lambda - \lambda^2)^{3/2}}.
\end{align}
\end{proposition}
We are now in position to state the following ``existence and uniqueness" theorem.  
\begin{theorem}
\label{th:L}
Under the assumptions of Proposition~\ref{prop:boundgi}, the solution  $(q(t),v(t))$ of (\ref{eq:Nbodyq})--(\ref{eq:icond}) is uniquely determined as a holomorphic function of  $t$ in the closed disk 
\begin{equation*}
D_{\lambda}(q^0,v^0) = \left\{ t \in \mathbb{C}\ : \ |t| \leq L(q^0,v^0,\lambda)^{-1} \right\},
\end{equation*}
where $\displaystyle L(q,v,\lambda) = \max_{1\leq i < j \leq N} L_{ij}(q,v,\lambda)$, with
\begin{equation*}
L_{ij}(q,v,\lambda) = 
\frac{||v_i-v_j||}{2\,\lambda \,\|q_i-q_j\|} +  
\sqrt{\left( \frac{||v_i-v_j||}{2\, \lambda\, \|q_i-q_j\|}\right)^2+
\frac{M_{ij}(q,\lambda)}{2\, \lambda\, \|q_i-q_j\|}},
\end{equation*}
and
\begin{equation}
\label{eq:Mij}
M_{ij}(q,\lambda) =  \eta(\lambda)\, (K_i(q) + K_j(q)).
\end{equation}
Furthermore,  for all $t \in D_{\lambda}(q^0,v^0)$,
\begin{align}
\label{eq:Lqaux}
 \| q_i(t)-q_j(t)-q^0_i+q^0_j\| &\leq |t| \, \|v_i^0-v_j^0\| + \frac{|t|^2}{2}\, M_{ij}(q^0,\lambda), \\
 \label{eq:Lvaux}
 \|v_i(t)-v_j(t)-v_i^0+v_j^0\| & \leq M_{ij}(q^0,\lambda) \, |t|,
\end{align}
which in particular implies that $q(t) \in \mathcal{U}_{\lambda}(q^0)$.
\end{theorem}

\begin{remark}
\label{rem:scale_invariance}
It is well known that, if $(q(t),v(t))$ is a solution of (\ref{eq:Nbodyq})--(\ref{eq:Nbodyv}) and $\nu >0$, then $(Q(t),V(t)) = (\nu^{-2/3}q(\nu\, t), \nu^{1/3}v(\nu \, t))$ is also a solution of  (\ref{eq:Nbodyq})--(\ref{eq:Nbodyv}). This implies that $\rho(\nu^{-2/3} q, \nu^{1/3} v) \equiv \nu^{-1} \, \rho(q,v)$. The lower estimate of the the radius of convergence $\rho(q,v)$ given by Theorem~\ref{th:L} is consistent with that scale invariance property, that is, $L(\nu^{-2/3} q, \nu^{1/3} v,\lambda) \equiv \nu \, L(q,v,\lambda)$. 
  \end{remark}

\begin{remark}
\label{rem:rho}
Theorem~\ref{th:L} provides, for each $\lambda \in (0,\sqrt{2}-1)$, a lower bound $1/L(q^*,v^*,\lambda)$ of the radius of convergence $\rho(q^*,v^*)$ at $t=t^*$ of any solution $q(t)$ of   (\ref{eq:Nbody2}). This means that the maximal real solution $q(t)$ of (\ref{eq:Nbodyq})--(\ref{eq:icond}) 
defined for $t \in (t_a,t_b)$  can be uniquely extended as a holomorphic function to the closed complex neighbourhood  
\begin{equation*}
\left\{ t  \in \mathbb{C}\ : \ |\mathrm{Im}(t)| \leq L(q(t^*),v(t^*),\lambda)^{-1}, \  t^* = \mathrm{Re}(t) \in (t_a,t_b) \right\}
\end{equation*}
of $(t_a,t_b)$, which is included in (\ref{eq:W}). In fact, Theorem~\ref{th:Painleve} can be seen as a corollary of Theorem~\ref{th:L}. As a matter of fact, according to  Remark~\ref{rem:W}, 
if $t_a >-\infty$, then $L(q(t),v(t),\lambda) \to \infty$ as $t  \downarrow t_a$. This implies that either $\min_{i,j} \|q_i(t) - q_j(t)\| \to 0$ or $\max_{i,j} \|v_i(t) - v_j(t)\| \to \infty$ as $t  \downarrow t_a$, but by virtue of (\ref{eq:Lvaux}), the later implies the former.  A similar argument holds when $t_b<+\infty$.
\end{remark}

\begin{remark}
{
For close enough encounters of, say  bodies $1$ and $2$,  we will have that $L(q,v,\lambda)=L_{1,2}(q,v,\lambda)$ and  the term $K_1(q)+K_2(q)$ in $L_{1,2}(q,v,\lambda)$ will be dominated by  $\frac{G\, (m_1 + m_2)}{\|q_1-q_2\|^2}$, and thus
\begin{equation}
\label{eq:L2body}
L(q,v,\lambda) \approx \frac{w}{2\,\lambda \,r} +  
\sqrt{\left( \frac{w}{2\, \lambda\, r}\right)^2+
\frac{\eta(\lambda) \, \mu}{2\, \lambda\, r^3}},
\end{equation}
where $w = ||v_1-v_2||$,  $r=\|q_1-q_2\|$,  and $\mu = G\, (m_1+m_2)$. 
As for the radius of convergence $\rho(q,v)$ of the $N$-body problem, it is the distance from $t=0$ to the nearest (possibly complex) time $t=t^*$ required for the two bodies to collide. For close enough encounters, such $t^*$ can be approximated by the collision times of the two-body problem obtained by ignoring the rest of the bodies.}

{It may be illustrative to compare the radius of convergence $\rho(q,v)$ and its estimate $L(q,v,\lambda)^{-1}$ for two extreme cases of collinear trajectories of the two-body problem: 
 \begin{itemize}
\item The case where the kinetic energy is much larger than the potential energy. In that case, 
the radius of convergence $\rho(q,v)$ is precisely the time needed for the two bodies to collide, which can be approximated by the inertial collision time $r/w$. On the other hand,
$L(q,v,\lambda)^{-1} \approx \lambda\, \frac{r}{w}$.
\item The free fall case, where $w=0$, the radius of convergence is again the time needed for the two bodies to collide, which can be seen to be $\rho(q,v) = \frac{\pi}{2}\, \sqrt{\frac{r^3}{2\mu}}$, while  $L(q,v,\lambda)^{-1} =   
\frac{2\, \lambda}{\sqrt{\lambda\, \eta(\lambda)}} \, \sqrt{\frac{r^3}{2 \, \mu}}$.  
\end{itemize}
}
\end{remark}

\begin{remark}
By virtue of Proposition~\ref{prop:boundgi}, the second order derivative of the solution $q(t)=(q_1(t),\ldots,q_N(t))$ of (\ref{eq:Nbodyq})--(\ref{eq:icond}) for $t \in D_{\lambda}(q^0,v^0)$ can be bounded as 
\begin{equation*}
\|q_i^{(2)}(t)\| = \|g_i(q(t)\| \leq \eta(\lambda)\, K_i(q^0), 
\end{equation*}
and Cauchy estimates give immediately 
\begin{equation*}
\|q_i^{(n)}(0)\| \leq \eta(\lambda)\, (n-2)! \, K_i(q^0) \, L(q^0,v^0,\lambda)^{n-2}, \quad n \geq 3.
\end{equation*}
\end{remark}

\begin{remark}
\label{rem:alpha}
Theorem~\ref{th:L} can be extended in a straightforward way to the more general case of $N$-body problems where  each pair of particles is  attracted with a force of magnitude $G m_i m_j \|q_i-q_j\|^{-\alpha-1}$, with $\alpha>0$. It is indeed enough to generalize  (\ref{eq:Ki}) in Proposition \ref{prop:boundgi}  as
\begin{equation}
\label{eq:Kalpha}
K_i(q) = \sum_{\substack{j=1 \\ j\neq i}}^N  \frac{G\, m_j}{||q_i - q_j||^{\alpha+1}}, \qquad 
\eta(\lambda) = \frac{1+\lambda}{(1 - 2 \lambda - \lambda^2)^{(\alpha+2)/2}}.
\end{equation}
\end{remark}

\subsection{Uniform  global time-renormalization  of $N$-body problems}
\label{ss:renormalization}
Our aim in this subsection is to determine an appropriate positive real analytic function $s(q,v)$ that relates the physical time $t$ in (\ref{eq:Nbodyq})--(\ref{eq:Nbodyv}) to a fictitious time-variable $\tau$ by 
\begin{equation}
\label{eq:s(q,v)}
\frac{d\tau}{d t} = s(q(t),v(t))^{-1}, \quad \tau(0)=0,
\end{equation}
such that equations (\ref{eq:Nbodyq})--(\ref{eq:Nbodyv}) expressed in terms of $\tau$ read
\begin{equation}
\label{eq:Nbodytau}
\begin{array}{lll}
 \displaystyle    \frac{d \hat q_i}{d \tau} &= s(\hat q,\hat v)\, \hat v_i, & i=1, \ldots, N, \\
     && \\
 \displaystyle \frac{d \hat v_i}{d \tau} &= \displaystyle s(\hat q,\hat v)\,  \sum_{j\neq i} \frac{G\, m_j}{||\hat q_i - \hat q_j||^3}(\hat q_j-\hat q_i),& i=1, \ldots, N.
\end{array}
\end{equation}
As is well-known, solutions $(q(t),v(t))$ of (\ref{eq:Nbodyq})--(\ref{eq:icond}) and $(\hat q(\tau),\hat v(\tau))$ of (\ref{eq:Nbodytau}) with initial condition
\begin{equation}
\label{eq:icondtau}
\hat q(0) = q^0, \quad \hat v(0)=v^0, 
\end{equation}
are related by $q(t) = \hat q(\theta(t)), \, v(t) = \hat v(\theta(t))$, where $\theta:(t_a,t_b) \to \R$ is given by (\ref{eq:theta}).
\\ \\
Before stating next proposition, recall from Theorem~\ref{th:L} that the solution $(q(t),v(t))$ of (\ref{eq:Nbodyq})--(\ref{eq:icond})  is such that  $(q(t),v(t)) \in \mathcal{U}_{\lambda}(q^0) \times  \mathcal{V}_{\lambda}(q^0,v^0)$ for all $t \in D_{\lambda}(q^0,v^0)$, where
\begin{align}
\label{eq:Vlambda}
\mathcal{V}_{\lambda}(q^0,v^0) &= 
\left\{ v\in \mathbb{C}^{6N}\ : 
 \|v_i-v_j-v_i^0+v_j^0\| \leq M_{ij}(q^0,\lambda)/L(q^0,v^0,\lambda)\right\}.
\end{align}

\begin{proposition}
\label{prop:thsaux}
Let $s(q,v)$ be a positive real analytic function in  $(\mathbb{R}^{3N}\backslash\Delta) \times \mathbb{R}^{3N}$.
Given $q^0 \in \mathbb{R}^{3N}\backslash\Delta$, $v^0 \in  \mathbb{R}^{3N}$, assume that there exists $\lambda \in (0,\sqrt{2}-1)$ and $\delta \in (0,1)$ such that $s(q,v)$ can be holomorphically extended to $\mathcal{U}_{\lambda}(q^0) \times  \mathcal{V}_{\lambda}(q^0,v^0)$, and such that 
\begin{equation}
\label{eq:delta}
\forall (q,v) \in \mathcal{U}_{\lambda}(q^0) \times  \mathcal{V}_{\lambda}(q^0,v^0), \quad 
| s(q,v)^{-1} -s(q^0,v^0)^{-1} | \leq \delta \, s(q^0, v^0)^{-1}.
\end{equation}
Then,  the solution $(\hat q(\tau),\hat v(\tau))$ of (\ref{eq:Nbodytau})--(\ref{eq:icondtau}) is uniquely defined as a holomorphic function of $\tau$ in the disk
\begin{equation}
\label{eq:taudisk}
\{ \tau \in \mathbb{C}\ :  \  |\tau| \leq (1 -\delta)\, s(q^0,v^0)^{-1}\,  L(q^0,v^0,\lambda)^{-1}\}.
\end{equation}
\end{proposition}
A positive real-analytic function $s(q,v)$ satisfying (\ref{eq:delta}) for some $\lambda$ and $\delta$, and such that $s(q,v)^{-1}L(q,v,\lambda)^{-1}$ is bounded from below for all  $(q,v) \in \mathbb{R}^{3N}\backslash\Delta \times \mathbb{R}^{3N}$ is thus a uniform global time-renormalization in the sense of Definition \ref{def:timereg}, with 
\begin{equation*}
\beta = (1-\delta) \inf_{q \in \mathbb{R}^{3N}\backslash\Delta, \ v \in  \mathbb{R}^{3N}} s(q,v)^{-1}L(q,v,\lambda)^{-1}.
\end{equation*}
Motivated by that, we thus seek a real-analytic function bounding $L(q,v,\lambda)$ from above for $(q,v)$ in $(\mathbb{R}^{3N}\backslash\Delta) \times \mathbb{R}^{3N}$. Such a function can be obtained by first observing that 
\begin{equation*}
L(q,v,\lambda) = \max_{1 \leq i < j \leq N} L_{ij}(q,v,\lambda) \leq 
\sqrt{
\sum_{1 \leq i < j \leq N} L_{ij}(q,v,\lambda)^2, 
}
\end{equation*}
and second noticing, as a consequence 
of the simple inequality 
$$
\forall (t_a,t_b) \in \R_+^2, \quad \frac{a}{2} + \sqrt{\left(\frac{a}{2}\right)^2 + \frac{b}{2}} \leq \sqrt{a^2+b},
$$
that
\begin{equation*}
L(q,v,\lambda) \leq \lambda^{-1} \sqrt{
\sum_{1 \leq i < j \leq N}
\left( \frac{||v_i-v_j||}{||q_i-q_j||}\right)^2+
\lambda \, \eta(\lambda)\,  \sum_{1 \leq i < j \leq N}
 \frac{K_i(q) + K_j(q)}{||q_i-q_j||}}.
 \end{equation*}
 It is thus natural to pick up the following candidate for $s(q,v)$
 \begin{equation}
 \label{eq:s}
s(q,v) = \left(
\sum_{1 \leq i < j \leq N}
 \frac{(v_i-v_j)^T (v_i-v_j)}{(q_i-q_j)^T(q_i-q_j)}+ \sum_{1 \leq i < j \leq N}
 \frac{K_i(q) + K_j(q)}{\sqrt{(q_i-q_j)^T(q_i-q_j)}}
\right)^{-1/2}
\end{equation}
where for each $i=1,\ldots,N$,
\begin{equation}
 K_i(q) = \sum_{\substack{k=1 \\ k \neq i}}^{N} \frac{G\, m_k}{(q_i-q_k)^T(q_i-q_k)}. 
\end{equation}

Let $\lambda_0 \approx 0.244204 < \sqrt{2}-1$ be the unique positive  zero of $\lambda\eta(\lambda)=1$. Clearly, if $(q^0,v^0) \in (\mathbb{R}^{3N}\backslash\Delta) \times \mathbb{R}^{3N}$, then we have 
\begin{equation}
\label{eq:sL}
\forall 0 \leq \lambda \leq \lambda_0, \; \forall (q^0,v^0) \in (\mathbb{R}^{3N}\backslash\Delta) \times \mathbb{R}^{3N}, \quad s(q^0,v^0)^{-1}L(q^0,v^0,\lambda)^{-1}\geq \lambda.
\end{equation}
 \begin{proposition}
\label{prop:thsaux2} 
There exists  $\lambda_* \in (0,\lambda_0)$ and $\delta:(0,\lambda_*) \to (0,1) $ such that the assumptions of Proposition~\ref{prop:thsaux}  hold for $s(q,v)$ given in (\ref{eq:s}) with  arbitrary $\lambda\in (0,\lambda_*)$ and $\delta=\delta(\lambda)$.
\end{proposition}
By combining Propositions~\ref{prop:thsaux} and \ref{prop:thsaux2} with (\ref{eq:sL}), we finally obtain the following.

\begin{theorem}
\label{th:s}
Consider $s(q,v)$ given by (\ref{eq:s}).
For arbitrary $q^0 \in \mathbb{R}^{3N}\backslash\Delta$, $v^0 \in \mathbb{R}^{3N}$,
the solution $(\hat q(\tau),\hat v(\tau))$ of (\ref{eq:Nbodytau})--(\ref{eq:icondtau})  is uniquely defined as a holomorphic function of the complex variable $\tau$ in the strip 
$$
\{\tau \in \mathbb{C}\ : \ |\mathrm{Im}(\tau)| \leq 0.0444443\}
$$ 
In other words,  $s(q,v)$ given by (\ref{eq:s}) is a time-renormalization function for 
$\beta=0.0444443$. 
\end{theorem}

\begin{remark}
Proposition~\ref{prop:thsaux}  holds with $\lambda_*=0.0988424$ as it can been seen in its proof, given in Section~\ref{sect:proofs} below. The constant $\beta= 0.0444443$ is 
obtained as the maximum (attained at $\lambda= 0.0694156$) of $(1 - \delta(\lambda))\, \lambda$ subject to the constraint $0<\lambda<\lambda_*$.
\end{remark}

\begin{remark}
Theorem~\ref{th:s} can be extended in a straightforward way  to the general case of $N$-body problems considered in Remark~\ref{rem:alpha} with $K_i(q)$ given by (\ref{eq:Kalpha})  and a different width $2\beta$ of the strip.
\end{remark}


\begin{remark}
As a consequence of Theorem~\ref{th:s}, we get a globally convergent series expansion in powers of a new variable $\sigma$, related to $\tau$ with
the conformal mapping
\begin{equation*}
\tau \mapsto \sigma =\frac{\exp(\frac{\pi}{2\beta} \tau)-1}{\exp(\frac{\pi}{2\beta} \tau)+1}.
\end{equation*}
that maps the strip $\{\tau \in \mathbb{C}\ : \ |\mathrm{Im}(\tau)| \leq \beta\}$ into the unit disk. 
This is closely related to Sundman's result~\cite{sundman} for the 3-body problem as well as Wang's results~\cite{wang} for the general case of $N$-body problems. It is worth emphasizing that, in contrast with Wang's solution,  our approach remains valid in the limit where  $\displaystyle \min_{1\leq i  \leq N} m_i/M \to 0$ with $M=\sum_{1 \leq i \leq N} m_i$.

\end{remark}

\subsection{Alternative time-renormalization functions}\label{ss:altfcn}


A computationally simpler alternative of (\ref{eq:s}) can be derived by observing that for each $1\leq i < j \leq N$,
\begin{equation*}
K_i(q) +K_j(q) \leq \sum_{1\leq k <\ell \leq N} \frac{G\, (m_{k} + m_{\ell})}{\|q_k-q_\ell\|^2}.
\end{equation*}
From that, it is straightforward to check that the proofs of Proposition~\ref{prop:thsaux2} and Theorem~\ref{th:s} are also valid for the time-renormalization function  
\begin{equation}
\label{eq:s2}
s(q,v) = \left(
\sum_{1 \leq i < j \leq N}
\left( \frac{||v_i-v_j||}{||q_i-q_j||}\right)^2+ A(q)
\sum_{1 \leq i < j \leq N}
 \frac{G\, (m_i + m_j)}{\|q_i-q_j\|^2} 
\right)^{-1/2}
\end{equation}
where
\begin{equation*}
A(q) = \sum_{1 \leq i < j \leq N}
 \frac{1}{\|q_i-q_j\|}. 
\end{equation*}

Another alternative time-renormalization function $s(q,v)$ can be obtained as follows: It is apparent from the proof of Theorem~\ref{th:L} given in Section~\ref{sect:proofs} that its conclusion remains true if $M_{ij}(q,\lambda)$ in (\ref{eq:Mij}) is replaced by any  upper bound of
\begin{equation}
\label{eq:supgij}
\sup_{q \in \mathcal{U}_{\lambda}(q^0)} \|g_i(q) - g_j(q)\|
\end{equation}
in {\em lieu} of  $\eta(\lambda)(K_i(q^0) + K_j(q^0))$. E.g., by writing  for any $1\leq i < j \leq N$ as 
\begin{align*}
g_i(q)-g_j(q)= \frac{G \, (m_i+m_j)}{ D(q_i,q_j)} (q_j-q_i) 
+ \sum_{\substack{k=1 \\ k \neq i, j}}   
\left( \frac{G \, m_k}{D(q_i,q_k)}\, (q_k-q_i) \right.
 \left. -  \frac{G \, m_k}{D(q_j,q_k)}\, (q_k-q_j) \right)
\end{align*}
where $D(q_i,q_j)=((q_i-q_j)^T (q_i-q_j))^{3/2}$ and accordingly for indices $j$ and $k$, 
it can be proven that there exists $\lambda_0>0$ such that 
\begin{align*}
\forall \lambda \in (0,\lambda_0), \quad \sup_{q \in \mathcal{U}_{\lambda}(q^0)} \|g_i(q) - g_j(q)\|
 &\leq \tilde{M}_{ij}(q^0,\lambda), 
 \end{align*}
 where
 \begin{align*}
  \tilde{M}_{ij}(q^0,\lambda) &= 
  \eta(\lambda) \, \frac{G\, (m_i+m_j)}{ \|q_i^0-q_j^0\|^2}  + \tilde{\eta}(\lambda)\, \sum_{\substack{k=1 \\ k \neq i, j}}  
\left( \frac{G\, m_k}{\|q_i-q_k\|^3} +  \frac{G\, m_k}{\|q_j-q_k\|^3} \, \right)\|q_i^0-q_j^0\|,
 \end{align*}
 for some  continuous function $\tilde{\eta}(\lambda)$ of $\lambda$ in $[0,\lambda_0)$.
This leads to consider $\tilde L$ instead of $L$ with 
\begin{equation*}
\tilde{L}(q,v,\lambda) = \max_{1 \leq i < j \leq N} \sqrt{\left( \frac{||v_i-v_j||}{\lambda\, \|q_i-q_j\|}\right)^2+
\frac{\tilde{M}_{ij}(q,\lambda)}{\lambda\, \|q_i-q_j\|}},
\end{equation*}
Then, noticing that 
\begin{equation*}
  \tilde{M}_{ij}(q^0,\lambda) \leq \tilde{\eta}(\lambda)\, \sum_{1 \leq k < l \leq N} \frac{G\, (m_k + m_l)}{\|q_k-q_l\|^3}, \end{equation*}
this suggests to consider
\begin{equation}
\label{eq:s3}
s(q,v) = \left(
\kappa\, \sum_{1 \leq i < j \leq N}
\left( \frac{||v_i-v_j||}{||q_i-q_j||}\right)^2+ 
\sum_{1 \leq i < j \leq N}
 \frac{G\, (m_i + m_j)}{\|q_i-q_j\|^3} 
\right)^{-1/2}
\end{equation}
for a suitable $\kappa>0$
as an alternative time-renormalization function. Indeed, it is possible to prove, in the vein  of Theorem~\ref{th:s}, that $s(q,v)$ given by (\ref{eq:s3}) (for some $\kappa>0$) is a uniform global time-renormalization function  (in the sense of Definition \ref{def:timereg}) for some value of $\beta>0$ depending on $\kappa >0$. 




\begin{remark}
Time-renormalization  functions that do not depend on the velocities $v_i$ are required for some numerical integrators. Such functions can be obtained by  replacing in (\ref{eq:s}), in (\ref{eq:s2}) or in (\ref{eq:s3}),  each term of the form $\|v_i-v_j\|^2$ by 
its upper bound 
$$\|v_i-v_j\|^2 \leq 2 (m_i^{-1/2} + m_j^{-1/2})^2 \, T(v)  =  2 (m_i^{-1/2} + m_j^{-1/2})^2 \, (E_0 + U(q)),$$ 
where 
 \begin{equation*}
T(v) = \frac12\, \sum_{i=1}^N m_i\, \|v_i\|^2, \quad
U(q) = \sum_{1 \leq i < j \leq N} \frac{G\, m_i\, m_j}{\|q_i-q_j\|}.
\end{equation*}
Note however that such a time-renormalization is no longer valid  in the limit when one of the masses vanishes.

 As an alternative function that does not depend on the velocities, one can consider (\ref{eq:s2}) with $\kappa=0$ (already considered by K. Babadzanjanz  \cite{babadzanjanz} in 1979):
\begin{equation}
\label{eq:s4}
s(q) = \left(
\sum_{1 \leq i < j \leq N}
 \frac{G\, (m_i + m_j)}{\|q_i-q_j\|^3} 
\right)^{-1/2}.
\end{equation}
If the close encounter corresponds to the periastrom of a very eccentric elliptic orbit or a parabolic  orbit, then $\frac{G\, (m_i + m_j)}{\|q_i-q_j\|^3}$ and $\frac12\, \left( \frac{||v_i-v_j||} {||q_i-q_j||}\right)^2$ will be of comparable magnitudes, and hence (\ref{eq:s4}) will behave similarly to (\ref{eq:s3}). 
However, that argument fails for binary close encounters where the inertial term $\left(\|v_i-v_j\|/(\|r_i-r_j\|)\right)$ dominates over the gravitational term. Actually, (\ref{eq:s4}) is not a uniform global time-renormalization function in the sense of Definition~\ref{def:timereg} (see second example in Section~\ref{sect:numer}).
\end{remark}

\begin{remark}
Some numerical integrators for time-renormalized equations require that $s(q,v)$ only depend on the potential function $U(q)$~\cite{mikkola,tremaine}.
In that case,  it makes sense to choose\footnote{
This is the choice of time-renormalization function adopted in~\cite{wang} to change the time-variable,  as an intermediate step to get the globally convergent series expansion of the solutions of the $N$-body problem} 
$s(q) = U(q)^{3/2}$, which is expected to behave asymptotically as (\ref{eq:s4})
near binary collisions. 
 It should be noted~(see \cite{laskar2019}) that this may perform poorly for systems with small mass ratios.
 \end{remark}
 
\begin{remark}
All time-renormalization functions introduced so far are scale invariant in the sense of~\cite{budd,blanes}, that is, $s(\nu^{-2/3}q, \nu^{1/3} v) = \nu^{-1} s(q,v)$. This implies that, if $(\hat q(\tau), \hat v(\tau), t(\tau))$ is a solution of the system consisting on the equations (\ref{eq:Nbodytau}) together with $\frac{d}{d\tau} t = s(\hat q, \hat v)$, then, for each $\nu>0$,  $(\nu^{-2/3} \hat q(\tau), \nu^{1/3} \hat v(\tau), \nu\,  t(\tau))$ is also a solution of that system.
\end{remark}

\section{Proofs of main results} 
\label{sect:proofs}

Existence and uniqueness of solutions of initial value problems of ordinary differential equations in the complex domain can be found in the book of E. Hille \cite{hille}, in particular, Theorem 2.3.1 (page 48) and its generalization to the vector case, referred to as Theorem 2.3.2 in \cite{hille}.  The technique of  proof  in Theorem 2.3.1 (based on Picard iteration) is the one we adopted in the following. \\ \\
\begin{proof}[Proof of Theorem~\ref{th:L}]
The solution $q(t)$ of equations (\ref{eq:Nbodyq})--(\ref{eq:icond}),  as long as it exists as a holomorphic function of the complex variable $t$, satisfies 
$$
q_i(t)=q_i^0+ t\, v_i^0 + \int_0^t \int_0^{\sigma} g_i(q(r)) dr d\sigma, \quad i=1,\ldots,N.
$$
We define the Picard iteration for $q_i(t)$, $i=1,\ldots,N$,  as follows:
\begin{align*}
n=0: & \quad q_i^{[0]}(t) = q_i^0 + t \, v_i^0, \\
n \geq 1: &  \quad q_i^{[n]}(t) = q_i^0 + t \, v_i^0+ \int_0^t \int_0^{\sigma} g_i(q^{[n-1]}(r)) dr d\sigma.
\end{align*}
We will prove by induction that for all $n \in \N$
\begin{enumerate}
\item[(i)] $q^{[n]}$ is well defined and holomorphic  on $D_{\lambda}(q^0,v^0)=\{t \in \C, |t| \leq L(q^0,v^0,\lambda)^{-1}\}$,
\item[(ii)] for all $t \in D_{\lambda}(q^0,v^0)$, $q^{[n]}(t) \in \mathcal{U}_{\lambda}(q^0)$.
\end{enumerate}
Note that, most importantly,  $t \in D_{\lambda}(q^0,v^0)$ if and only if 
\begin{equation}
\label{eq:taux}
|t|\, \|v_i^0-v_j^0\| + \frac{|t|^2}{2}\, M_{ij}(q^0,\lambda) \leq \lambda \|q_i^0-q_j^0\|, \quad 1\leq i < j \leq N.
\end{equation}
Given that statements (i) and (ii) hold for $n=0$, assume that they hold for some  $n-1 \geq 0$. We have in particular that,
\begin{equation*}
\forall 1\leq i < j \leq N, \; \forall t \in D_{\lambda}(q^0,v^0), \quad \|g_i(q^{[n-1]}(t)) - g_j(q^{[n-1]}(t))\| \leq M_{i j}(q^0,\lambda).
\end{equation*}
Then, as a double integral of a holomorphic function, which is itself holomorphic as the composition of two holomorphic functions, $q^{[n]}(t)$ is also holomorphic on $D_{\lambda}(q^0,v^0)$. In order to prove (ii), we estimate for $i=1,\ldots,N$
\begin{align*}
\|q_i^{[n]}(t)-q_j^{[n]}(t)-(q_i^0 - q_j^0)\| &\leq |t|\, \|v_i^0-v_j^0\| +\int_0^t \int_0^{\sigma} M_{i j}(q^0,\lambda) dr d\sigma\\
&\leq   |t|\, \|v_i^0-v_j^0\| + \frac{|t|^2}{2} \, M_{i j}(q^0,\lambda),
\end{align*}
which, together with (\ref{eq:taux}), implies that $q^{[n]}(t) \in \mathcal{U}_{\lambda}(q^0)$.

 Now, in order to prove the convergence 
of the sequence $q^{[n]}(t)$ in the Banach space of holomorphic functions on $\mathcal{U}_{\lambda}(q^0)$ with norm $\|\cdot\|_*{}$
\begin{equation*}
\|q\|_{*} = \max_{1\leq i \leq N} \|q_i\|,
\end{equation*} 
we write 
$$
q(t) - q^{[0]}(t) = \sum_{n=1}^\infty (q^{[n]}(t)-q^{[n-1]}(t))
$$
and show that the series is absolutely convergent for all $t \in D_{\lambda}(q^0,v^0)$. To this aim, we write 
\begin{align*}
 \|q_i^{[1]}(t)-q_i^{[0]}(t)\| \leq  \frac{|t|^2}{2} \, M_{i}(q^0,\lambda), \quad
 \|q_i^{[2]}(t)-q^{[1]}(t)\| \leq \ell_i(q^0,\lambda) \, \eta(\lambda) K_{i}(q^0) \frac{|t|^4}{4!}
\end{align*}
where $\ell_i(q^0,\lambda)$ is a Lipschitz constant of $g_i$ in $\mathcal{U}_{\lambda}(q^0)$.
More generally, we obtain
\begin{align*}
 \|q_i^{[n]}(t)-q_i^{[n-1]}(t)\| \leq \frac{ M_{i}(q^0,\lambda) }{\ell_i(q^0,\lambda)} \frac{(\sqrt{\ell_i(q^0,\lambda)} |t|)^{2n}}{(2n)!}.
\end{align*}
Note that we have relied on the fact that $q^{[n]}(t) \in \mathcal{U}_{\lambda}(q^0)$ for $t \in D_\lambda(q^0,v^0)$ in order to bound $g_i$ by $\eta(\lambda) K_{i}(q^0)$ and to use the Lipschitz constant on $\mathcal{U}_{\lambda}(q^0)$. 
As a consequence, the series is absolutely convergent with 
$$
\sum_{n=1}^\infty \|q^{[n]}(t)-q^{[n-1]}(t)\| \leq \frac{ M_{i}(q^0,\lambda) }{\ell_i(q^0,\lambda)} \left(\cosh(\sqrt{\ell_i(q^0,\lambda)} |t|)-1\right)
$$
and this proves the existence of a solution on $D_{\lambda}(q^0,v^0)$, holomorphic at least on the interior of $D_{\lambda}(q^0,v^0)$ and continuous on the boundary. The uniqueness can be obtained through the use of the Lipschitz constant again. 
\end{proof} \\ \\
%
\begin{proof}[Proof of Proposition~\ref{prop:thsaux}] 
First observe that  (\ref{eq:delta}) implies that, for all $(q,v) \in \mathcal{U}_{\lambda}(q^0) \times  \mathcal{V}_{\lambda}(q^0,v^0)$,
\begin{equation*}
(1-\delta) \, s(q^0,v^0)^{-1} \leq |s(q,v)^{-1}| \leq (1+\delta) \, s(q^0,v^0)^{-1},
\end{equation*}
so that, in particular, 
$s(q,v)^{-1}$ is holomorphic in $\mathcal{U}_{\lambda}(q^0) \times  \mathcal{V}_{\lambda}(q^0,v^0)$.
We thus have that  $s(q(t),v(t))^{-1}$ is a holomorphic function of the complex variable $t$, so that $\theta(t)$ (see Definition \ref{def:timereg}) is a well-defined holomorphic function in the disk
\begin{equation*}
D_{\lambda}(q^0,v^0) = \left\{ t \in \mathbb{C}\ : \ |t| \leq L(q^0,v^0,\lambda)^{-1} \right\}.
\end{equation*} 
Furthermore, if $t_1,t_2 \in D_{\lambda}(q^0,v^0) $, 
\begin{equation*}
|\theta(t_2) -\theta(t_1)| =
 \left| 
 \int_{t_1}^{t_2} s(q(t),v(t))^{-1}\, dt 
 \right| \geq (1 -\delta)\, s(q^0,v^0)^{-1}\, |t_2-t_1|.
\end{equation*}
Whence it follows that $\theta$ is injective in $D_{\lambda}(q^0,v^0) $. If $t^* \in \mathbb{C}$ is on the boundary of $D_{\lambda}(q^0,v^0) $, then
\begin{align*}
|\theta(t^*)| &=
 \left| 
 \int_{0}^{t^*} s(q(t),v(t))^{-1}\, dt
 \right| \\
 &\geq (1 -\delta)\, s(q^0,v^0)^{-1}\, |t^*| \\
 &= (1 -\delta)\, s(q^0,v^0)^{-1}\, L(q^0,v^0,\lambda)^{-1}.
\end{align*}
Hence, the image by $\theta$ of $D_{\lambda}(q^0,v^0) $ contains the disk (\ref{eq:taudisk}).  We conclude that $\theta^{-1}$ is well-defined and holomorphic in the disk (\ref{eq:taudisk}). 
 \end{proof} \\ \\
We next state without proof the following auxiliary lemma.
\begin{lemma}
\label{lem:tech2}
Consider $u \in \R^3$, $u \in \C^3$ such that $\|u-u^0\| \leq \lambda \, \|u^0\|$, with $\lambda \in (0,\sqrt{2}-1)$. The following estimates hold true
\begin{align*}
(i) \quad & |u^T u - \|u^0\|^2| \geq  (2\lambda+\lambda^2)\, \|u^0\|,\\
(ii) \quad & |u^T u| \geq  (1-2\lambda-\lambda^2)\, \|u^0\|,\\
(iii) \quad &  \left| (u^T u)^{-1} -  \|u^0\|^{-2} \right| \leq  \lambda\, \alpha(\lambda)\, \|u^0\|^{-2}, \\
(iv) \quad & 
\left| (u^T u)^{1/2}-  \|u^0\| \right|  \leq \lambda\, \beta(\lambda)\, \|u\|^0, \\
(v) \quad & \left|(u^T u)^{-1/2} -  \|u^0\|^{-1} \right| \leq \lambda\, \gamma(\lambda)\, \|u^0\|^{-1},
\end{align*}
where
\begin{align*}
\alpha(\lambda) = \frac{2+\lambda}{1-2\lambda-\lambda^2}, \quad 
\beta(\lambda) = \frac{2+\lambda}{1 + \sqrt{1-2\lambda -\lambda^2}}, \quad 
\gamma(\lambda) =  \frac{\beta(\lambda)}{ \sqrt{1-2\lambda -\lambda^2}}.
\end{align*}
\end{lemma}
In what follows, we use the notations $r_{ij} = q_i-q_j$, $w_{ij} = v_i-v_j$, and likewise with zero superscript. \\
\begin{proof}[Proof of Proposition~\ref{prop:thsaux2}] 
Under the assumption  that $(q,v) \in \mathcal{U}_{\lambda}(q^0) \times \mathcal{V}_{\lambda}(q^0,v^0)$, 
 Lemma~\ref{lem:tech2} implies that, for $i=1,\ldots,N$, 
\begin{equation*}
| K_i(q) - K_i(q^0)| \leq \lambda\, \alpha(\lambda) \, K_i(q^0).
\end{equation*}
and
\begin{align*}
 \left| \frac{K_i(q) + K_j(q)}{(r_{ij}^T r_{ij})^{1/2}} - 
 \frac{K_i(q^0) + K_j(q^0)}{\|r_{ij}^0 \|}
 \right| &\leq
  \left| \frac{ K_i(q) - K_i(q^0) + K_j(q) - K_j(q^0) }{(r_{ij}^T r_{ij})^{1/2}}\right| \\ 
&+ (K_i(q^0) + K_j(q^0)) 
\left|
\frac{1}{(r_{ij}^T r_{ij})^{1/2}} -
\frac{1}{\|r_{ij}^0 \|}
\right|\\
&\leq 
\lambda\, \nu(\lambda)\,  \frac{K_i(q^0) + K_j(q^0)}{\|r_{ij}^0 \|},
\end{align*}
where
\begin{equation*}
\nu(\lambda) = \frac{\alpha(\lambda)}{\sqrt{1-2\lambda - \lambda^2}} + \gamma(\lambda).
\end{equation*}
On the other hand, $ \|w_{ij}-w_{ij}^0\| \leq M_{ij}(q^0,\lambda)/L(q^0,v^0,\lambda)$ implies that
\begin{equation*}
\|w_{ij}-w_{ij}^0\| \leq \frac{\lambda\, \|r_{ij}^0\| \, M_{ij}(q^0,\lambda)}{\|w_{ij}^0\|}   \, \mbox{ and } \, 
 \|w_{ij}-w_{ij}^0\|^2 \leq \lambda\, \|r_{ij}^0\| \, M_{ij}(q^0,\lambda).
\end{equation*}
This, together with  Lemma~\ref{lem:tech} gives
\begin{align*}
\left| w_{ij}^T w_{ij} - \|w_{ij}^0\|^2 \right|& \leq 3\,  \lambda\, \|r_{ij}^0\| \, M_{ij}(q^0,\lambda)\\
& = 3\,  \lambda\, \eta(\lambda)\, \|r_{ij}^0\|\, (K_i(q^0) + K_j(q^0)).
\end{align*}
Then, 
\begin{align*}
 \left|  \frac{ w_{ij}^T w_{ij}}{ r_{ij}^T r_{ij}} - \frac{\|w_{ij}^0\|^2}{\|r_{ij}^0\|^2} 
 \right| &\leq    \frac{\left| w_{ij}^T w_{ij} - \|w_{ij}^0\|^2 \right|}{ r_{ij}^T r_{ij}}  
  + \|w_{ij}^0\|^2\,  
\left(
\frac{1}{ r_{ij}^T r_{ij}} -
\frac{1}{\|r_{ij}^0 \|^2}
\right)\\
 &\leq \lambda\, \xi(\lambda)\,  \frac{K_i(q^0) + K_j(q^0)}{\|r_{ij}^0 \|}  
  + \lambda\, \alpha(\lambda)\,  
\frac{\|w_{ij}^0\|^2}{\|r_{ij}^0 \|^2},
\end{align*}
where
\begin{equation*}
\xi(\lambda) = \frac{3\, \eta(\lambda)}{(1-2\lambda - \lambda^2)^{2}}.
\end{equation*}
Hence, 
\begin{align*}
\left| s(q,v)^{-2} - s(q^0,v^0)^{-2} \right| &\leq \lambda\, \alpha(\lambda)\,  \frac{\|w_{ij}^0\|^2}{\|r_{ij}^0 \|^2} + 
 \lambda\, (\nu(\lambda) + \xi(\lambda)) \, \frac{K_i(q^0) + K_j(q^0)}{\|r_{ij}^0 \|}\\
& \leq  \lambda\, \mu(\lambda) \, s(q^0,v^0)^{-2},
\end{align*}
where
\begin{equation*}
\mu(\lambda) = \max(\gamma(\lambda), \nu(\lambda) + \xi(\lambda)) = \nu(\lambda) + \xi(\lambda).
\end{equation*}
Then, provided that $\lambda\, \mu(\lambda) < 1$ (i.e., provided that $\lambda < \lambda_0=0.0988424 < \sqrt{2}-1$),
\begin{align*}
| s(q,v)^{-2} | \geq (1 - \lambda\, \mu(\lambda)) \, s(q^0,v^0)^{-2}, \\
| s(q,v)^{-1} | \geq \sqrt{1 - \lambda\, \mu(\lambda)}\, s(q^0,v^0)^{-1}, \\
| s(q,v)^{-1} -s(q^0,v^0)^{-1} | \leq  \lambda\, \delta(\lambda)\, s(q^0,v^0)^{-1},
\end{align*}
where 
\begin{equation*}
\delta(\lambda) =  \frac{\mu(\lambda)}{1 + \sqrt{1 -  \lambda\, \mu(\lambda)}}.
\end{equation*}
\end{proof}


\section{Numerical experiments} \label{sect:numer}

We consider three examples, corresponding to the solution $(q(t),v(t))$ of (\ref{eq:Nbodyq})--(\ref{eq:icond}) for two 3-body problems and one 9-body problem in certain time intervals $t \in [0,T]$. For each example, we obtain several figures:
\begin{enumerate}
\item A comparison of the actual radius of convergence with the lower bound given by Theorem~\ref{th:L}. More precisely,
\begin{itemize}
\item we obtain accurate approximations $(q^k, v^k)$ of the solution $(q(t^k),v(t^k))$ for some discretization $0=t^0 < t^1 < \cdots < t^k < \cdots< t^n=T$ of the time interval $[0,T]$. This is done 
by numerically integrating the renormalized equations (\ref{eq:Nbodytau}), with the time-renormalization function (\ref{eq:s}), together with 
\begin{equation}
\label{eq:t}
\frac{d}{d\tau} t = s(\hat q, \hat v), \quad t(0) = 0.
\end{equation}
To that aim, we apply a Taylor integrator of order 30 with constant step-size $\Delta \tau$. More precisely, we compute the positions $q^k \in \R^{3N}$, velocities $v^k \in \R^{3N}$, and times $t^k \in \R$, corresponding to $\tau^k = k\, \Delta \tau$ for $k=1,2,\ldots$, in a step-by-step manner by considering the Taylor expansion of the solution $(\hat q(\tau), \hat v(\tau), t(\tau))$ at $\tau=\tau^k$, $k=0,1,2,\ldots$ truncated at order 30.
\item  the radius of convergence $\rho(q^k,v^k)$ of the Taylor expansion of the solution $q(t)$ centered  at $t=t^k$ is computed numerically. This is done by computing the Taylor expansion of the solution $q(t)$ centered at $t=t^k$, ($k=0,1,2,\ldots$) truncated at order 30, and numerically estimating the radius of convergence. 
\item  the radius of convergence $\rho(q^k,v^k)$ is displayed in logarithmic scale together with the lower bound $1/L(q^k,v^k,\lambda_0)$ (with $\lambda_0=0.244204$) of  $\rho(q^k,v^k)$.
\end{itemize}

\item   {We denote the time-renormalization function (23) as $s_1(q,v)$, (26) as $s_2(q,v)$, (28)
with $\kappa = 1$ as $s_3(q,v)$,  (29) as $s_4(q)$, and $s_0(q,v)\equiv1$. For each $j=0,1,2,3,4$, we denote $T_j$ the positive real number such that the time interval $[0,T]$ corresponds to the interval $[0,T_j]$ in $\tau$.
For each $j$, the maximal width of the strip (in the complex $\tau$-plane) around the integration interval $[0,T_j]$ where the solution $(\hat q(\tau), \hat v(\tau))$ is holomorphic is computed as the minimum of the radius of convergence of the Taylor expansion of $(\hat q(\tau), \hat v(\tau))$ for $\tau \in [0,T_j]$. According to Theorem~\ref{th:s}, for $j=1,2$,  the width of the strip must be greater than or equal to $2\beta =  0.0888886$.
For the purpose of fairly comparing the different time-renormalization functions $s_j$, the width of the strip corresponding to the time-renormalization function $s_j(q,v)$ is scaled by $T/T_j$. }

\item A comparison of the efficiency of numerical integration of the renormalized equations  (\ref{eq:Nbodytau}) with constant step-size with the numerical integration  of the original equations (\ref{eq:Nbodyq})--(\ref{eq:Nbodyv}) with standard adaptive step-size strategy.  For that purpose,  the embedded pair of explicit Runge-Kutta schemes of order 9(8) constructed by Verner~\cite{verner} is chosen,  implemented (and recommended for high precision computations) in the Julia package DifferentialEquations.jl~\cite{rackauckas}.   We first integrate the problem (without time-renormalization) with adaptive implementation of Verner's 9(8) embedded RK pair with relative and absolute tolerances of $10^{-14}$, which results in the application of, say, $n$ successful steps of the 9th order RK scheme. We then apply, for each time-renormalization function $s_i(q,v)$, $i=1,2,3,4$, 
the same 9th order RK scheme to the equations (\ref{eq:Nbodytau})  (together with (\ref{eq:t})) with constant time-step $\Delta \tau$ chosen so that also $n$ steps are required to integrate from $t=0$ to $t=T$. In all examples, we use high precision floating point arithmetic  (Julia's Bigfloat numbers with default precision, that is, floating point numbers with 256-bit mantissa  from GNU MPFR Library) in order to  minimize the influence of roundoff errors in the numerical results. {For a fair comparison of efficiency of, on the first hand, adaptive numerical integration  of the original equations and on the other hand, constant step-size numerical integration of renormalized equations, one should take into account the cost in CPU time of the evaluation of the renormalized equations versus the original equations. In practice, we have observed that the difference in CPU time for the case of the time-renormalization functions $s_j$ ($j=2,3,4$) is marginal (the cheapest case being $j=4$ followed by $j=3$). Indeed, the evaluation of such $s_j$ can be very efficiently implemented by adding a few elementary arithmetic operations inside the nested loop used to evaluate the original equations of motion.  The function $s_1$ is more costly, as a separate nested loop is required in this case. }
\end{enumerate}

\subsection{The Pythagorean 3-body problem} 
This is a planar 3-body problem first investigated numerically by Burrau (1913). It consists in three masses in the ratio 3:4:5 placed at rest at the vertices of a 3:4:5 right triangle. Victor Szebehely and C. Frederick Peters (1967) found, by numerically integrating the problem (using Levi-Civita renormalization near close encounters), that after several binary close encounters a stable binary is formed between the two heaviest bodies while the third one is ejected from the system.  

As in~\cite{Szebehely}, we consider the problem in adimensional units, with $G\, m_i = 5, 4, 3$,  initial positions $q^0_i =  (1,-1), (-2,-1), (1,3)$,  and zero initial velocities.  We compute the solution in the time interval $t \in [0,63]$. The trajectories of the three bodies are displayed in Figure~\ref{fig:PythagoreanRadius}.

\begin{figure}[t]
\begin{center}
\resizebox{!}{10cm}{\includegraphics{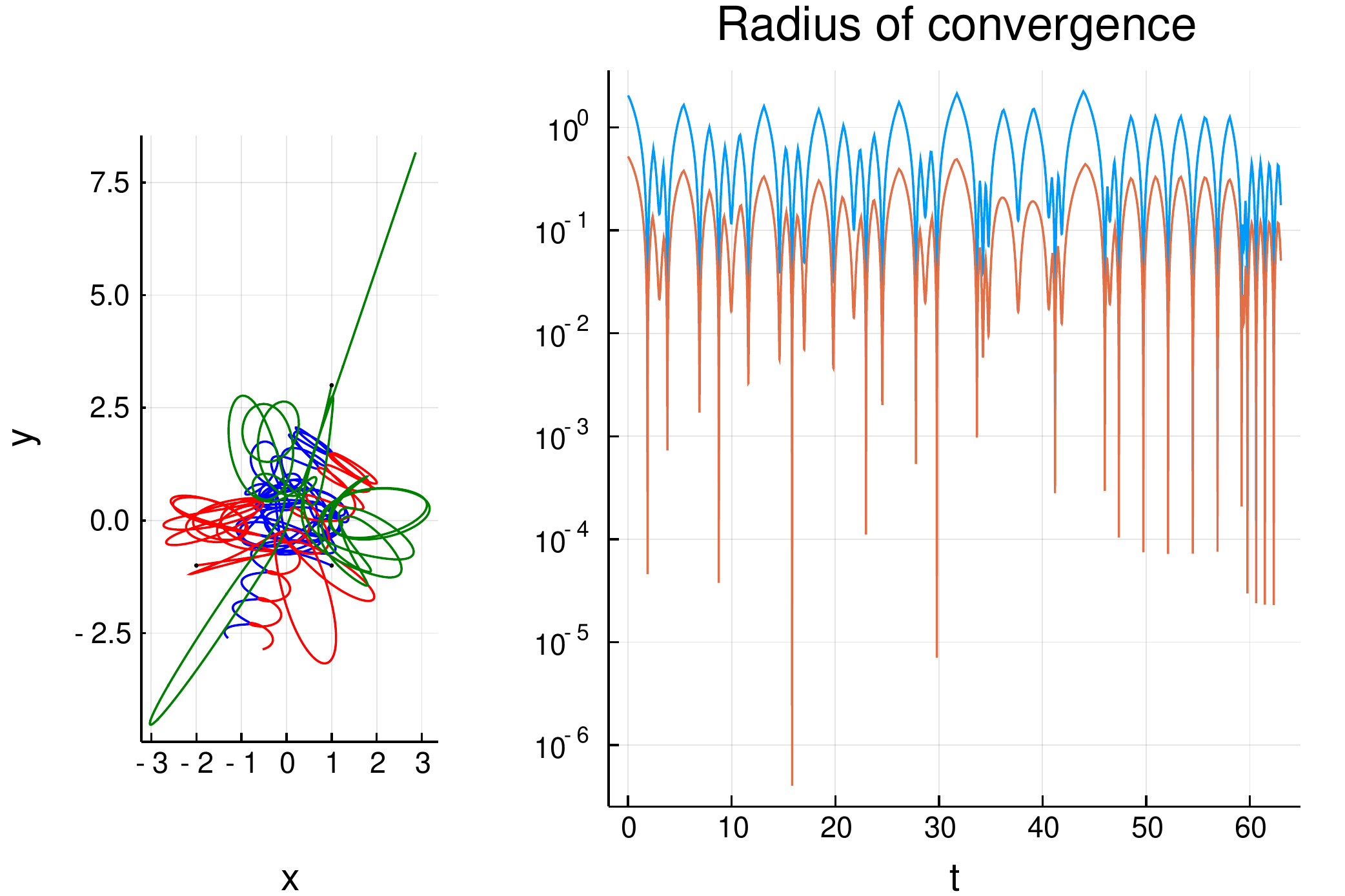}}
\end{center}
\caption{
\label{fig:PythagoreanRadius}
Trajectories of the three bodies in the Pythagorean 3-body problem (left). Values of radius of convergence $\rho(q^k,v^k)$ together with its lower bound $1/L(q^k,v^k,\lambda_0)$ for the Pythagorean 3-body problem (right)}
\end{figure}

\begin{figure}[t]
\begin{center}
\resizebox{!}{10cm}{\includegraphics{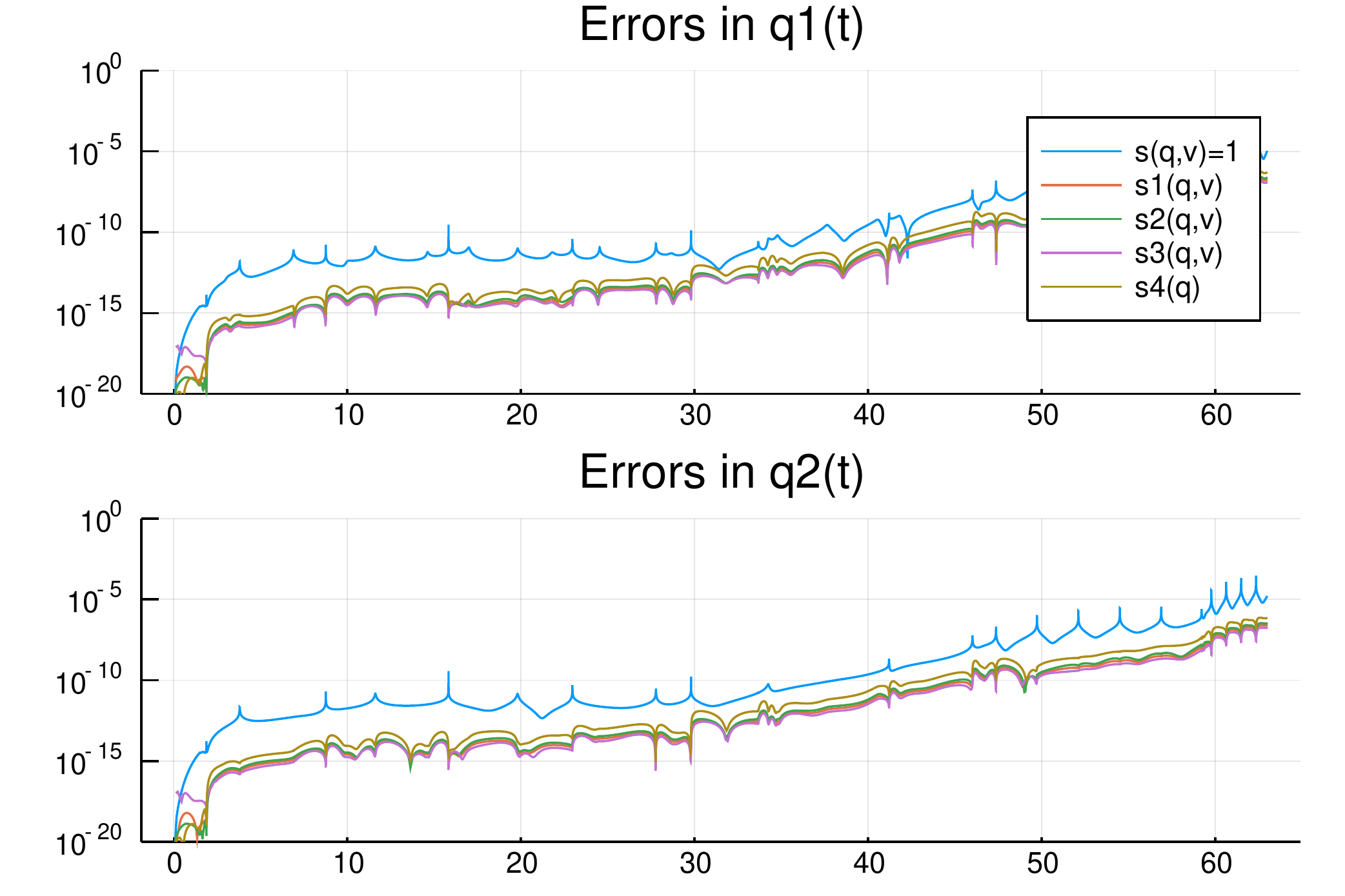}}
\end{center}
\caption{
\label{fig:PythagoreanQerr}
Evolution of position errors  for the Pythagorean 3-body problem for different numerical integrations. Constant step-size integration for $s_i(q,v)$, $i=1,2,3,4$, and adaptive integration for $s_0(q,v)=1$.
}
\end{figure}

In Figure~\ref{fig:PythagoreanRadius}, 
the  radius of convergence $\rho(q^k,v^k)$ is displayed in logarithmic scale together with the lower bound $1/L(q^k,v^k,\lambda_0)$ given by Theorem~\ref{th:L}. The minimum and maximum of the scaling factor $\rho(q^k,v^k)L(q^k,v^k,\lambda_0)$ are $3.41111$ and $7.9689$ respectively.

\begin{table}[!ht]
\caption[Pythagorean problem] 
{\small{Pythagorean problem}}
\label{tab:Pythagorean}       
\centering
{%
\begin{tabular}{  l l l l l l } 
\\
                 & $s_0(q,v)=1$     & $s_1(q,v)$ & $s_2(q,v)$ & $s_3(q,v)$ &  $s_4(q)$ \\ 
                 & (adaptive)      & & & & \\
 \hline
\\
 Width of strip    & $4.8 \times 10^{-6}$      & $2.06$    & $2.297$    & $1.835$  & $1.402$  \\  \\
Scaled width    & $4.8 \times 10^{-6}$      & $0.2532$    & $0.2716$    & $0.2359$  & $0.2891$  \\  \\
 Energy error    & $6 \times 10^{-12}$    & $9 \times 10^{-15}$  & $1.3 \times 10^{-14}$  & $6 \times 10^{-15}$  & $2.9 \times 10^{-14}$ \\ 
 \\  
   \hline
 \end{tabular}}
\end{table}

The widths of the strips around $[0,T]$ for the  time-renormalization functions $s_i(q,v)$, $i=1,2,3,4$ and its scaled versions,  together with  the width of the strip corresponding to $s_0(q,v)=1$,
are displayed in Table~\ref{tab:Pythagorean}. The corresponding maximum errors in energy are  also displayed in the same table.  The (non-scaled) widths of $s_i(q,v)$, $i=1,2$, are larger than their lower bound $2\beta=0.0888886$ given by Theorem~\ref{th:s}, as expected.
 The position errors $\|q_i(t^k)-q_i^k\|$ ($i=1,2$) for the different numerical integrations  are plotted in Figure~\ref{fig:PythagoreanQerr}.  A very similar plot (not shown) is obtained for $\|q_3(t^k)-q_3^k\|$.  A clear superiority of the constant step-size integrations of the time-renormalized equations over the adaptive numerical integration of the equations in physical time can be observed for all the considered choices of the time-renormalization functions.

\subsection{A binary-visitor problem} 
A binary with a mass ratio of $2:1$, eccentricity $e=0.9$, and semi-major axis $a=10$ is visited by a smaller third body with a $1:100$ mass ratio (with respect to the most massive body of the binary), which crosses their orbital plane  through their center of masses perpendicularly at velocity $\|v_3\|=100$ when the components of the binary are in their closest position. The third body has enough energy to avoid being captured by the binary.

\begin{figure}[t]
\begin{center}
\resizebox{!}{10cm}{\includegraphics{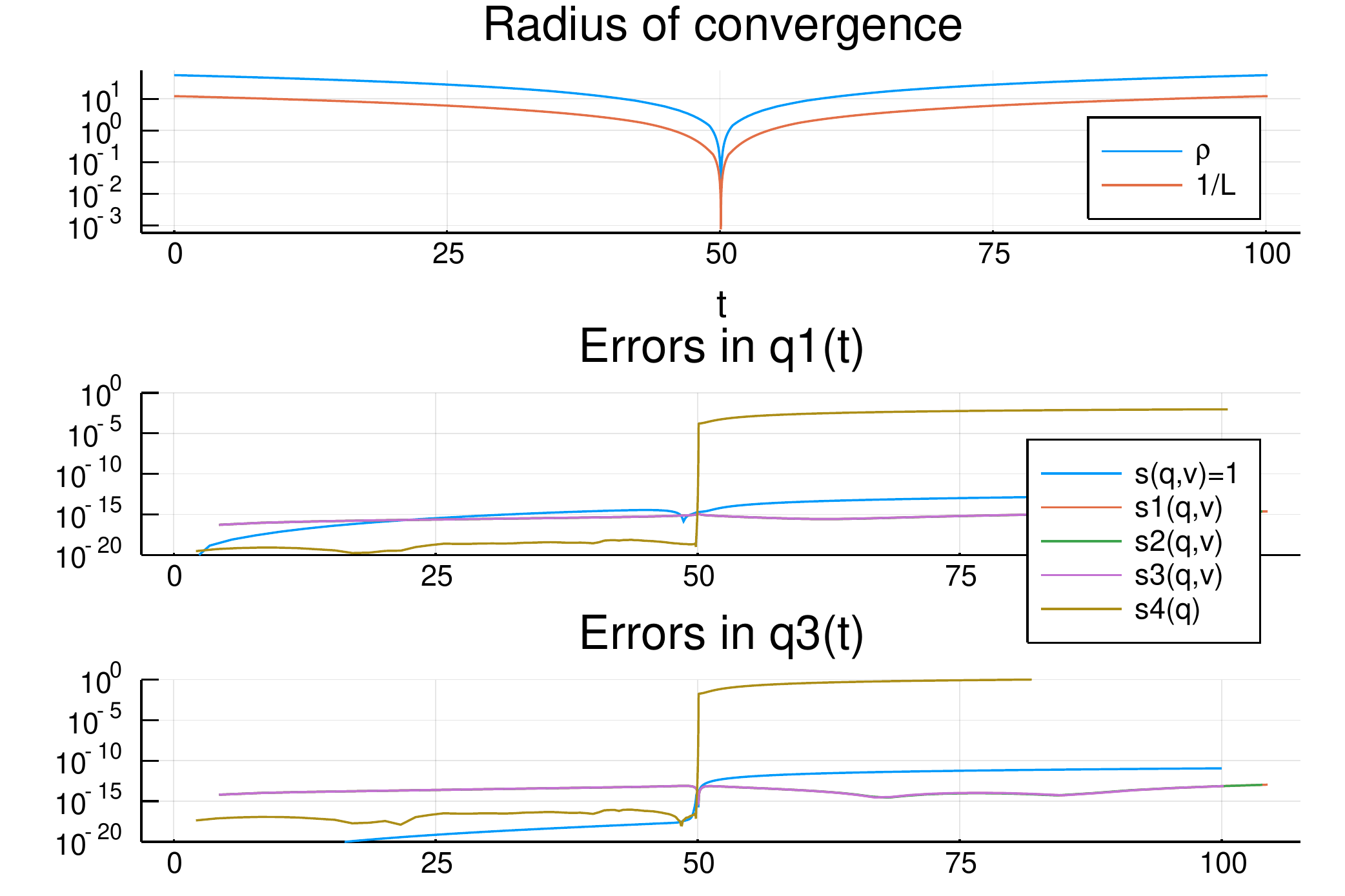}}
\end{center}
\caption{
\label{fig:radiusExample2}
Top: Values of radius of convergence $\rho(q^k,v^k)$ and its lower bound $1/L(q^k,v^k,\lambda_0)$ for the binary-visitor 3-body problem. Middle (resp. bottom): Evolution of position errors  of $m_1$ (resp. $m_3$) for the binary-visitor 3-body problem for different numerical integrations (constant step-size integration for $s_i(q,v)$, $i=1,2,3,4$, and adaptive integration for $s_0(q,v)=1$).}
\end{figure}

In Figure~\ref{fig:radiusExample2}, 
the  radius of convergence $\rho(q^k,v^k)$ is displayed in logarithmic scale together with the lower bound $1/L(q^k,v^k,\lambda_0)$ given by Theorem~\ref{th:L}. The minimum and maximum of $\rho(q^k,v^k)L(q^k,v^k,\lambda_0)$ are $4.32562$ and $5.47193282$ respectively.

\begin{table}[!ht]
\caption[A binary-visitor problem] 
{\small{A binary-visitor problem}}
\label{tab:binary-visitor}       
\centering
{%
\begin{tabular}{ l l l l l l } 
\\
                 & $s_0(q,v)=1$   & $s_1(q,v)$ & $s_2(q,v)$ & $s_3(q,v)$ &  $s_4(q)$ \\ 
                 & (adaptive)      & & & & \\
 \hline 
\\
 Width of strip    & $0.008$      & $ 3.2897$    & $3.2918$    & $3.2875$  & $0.0369$  \\  \\
Scaled width    & $0.008$      & $10.354$    & $10.354$    & $10.354$  & $0.4941$  \\  \\
 Energy error     & $5\times 10^{-15}$    & $6 \times 10^{-19}$  & $6 \times 10^{-19}$  & $6 \times 10^{-19}$  & $4 \times 10^{-6}$ \\ 
 \\  
   \hline
 \end{tabular}}
\end{table}

The widths of the strips around $[0,T]$ for the  time-renormalization functions  $s_i(q,v)$, $i=0,1,2,3,4$ and  their scaled versions, together with  the corresponding maximum errors in energy
are displayed in Table~\ref{tab:binary-visitor}.  In this example, one observes that the time-renormalization function $s_4(q)$ does not perform well. That ill-behavior is exacerbated if the velocity when the binary is crossed by the visitor at a higher velocity: the scaled width of the corresponding strip  is approximately $49.8/\|v_3\|$ as $\|v_3\| \uparrow \infty$. (For the equations in physical time, it is approximately $0.86/\|v_3\|$ as $\|v_3\| \uparrow \infty$.)

The position errors $\|q_i(t^k)-q_i^k\|$ ($i=1,3$) for the different numerical integrations are plotted in Figure~\ref{fig:radiusExample2}.  A clear superiority of the constant step-size integrations of the time-renormalized equations over the adaptive numerical integration of the equations in physical time can be observed  for $s_i(q,v)$, $i=1,2,3$ (their error curves are virtually identical). The very poor performance of $s_4(q)$ is in accordance with the narrowness of the strip noted in previous item.

\subsection{Solar system} 
We consider a 9-body model of the Solar System, with the Sun, Mercury, Venus, the Earth (the Earth-Moon barycenter), Mars, Jupiter, Saturn, Neptune, and Uranus. The initial values (and the product of the gravitational constant with their corresponding masses) are taken from DE430, Julian day (TDB) 2440400.5 (June 28, 1969). We consider a time interval of $2000$ days, that is, $t \in [0,T]$ with $T=2000$.

\begin{figure}[t]
\begin{center}
\resizebox{!}{7cm}{\includegraphics{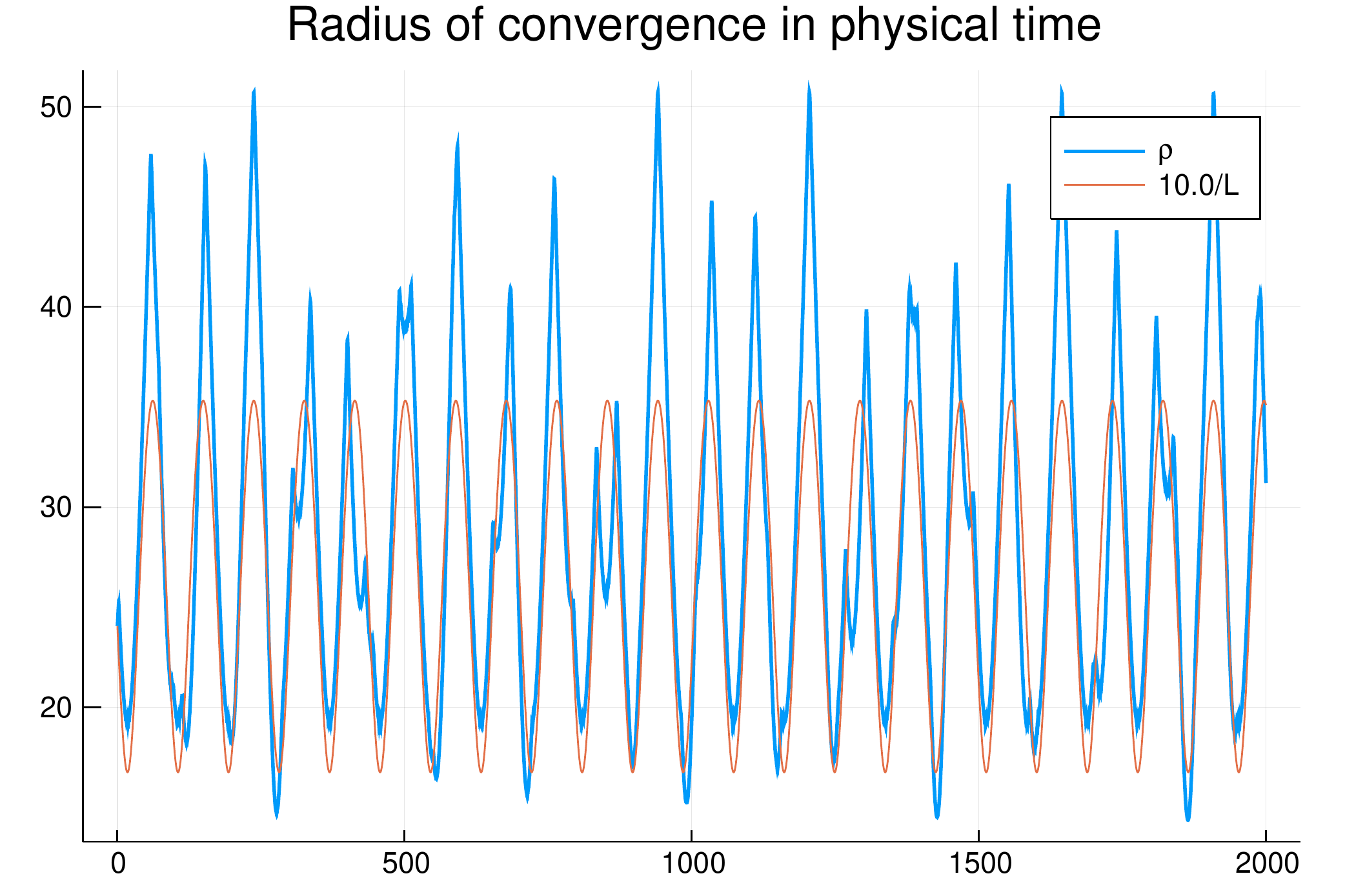}}
\end{center}
\caption{
\label{fig:radiusExample3}
Values of radius of convergence $\rho(q^k,v^k)$ together with its lower bound $1/L(q^k,v^k,\lambda_0)$ scaled by 10 for the Solar-System 9-body problem}
\end{figure}

\begin{figure}[t]
\begin{center}
\resizebox{!}{10cm}{\includegraphics{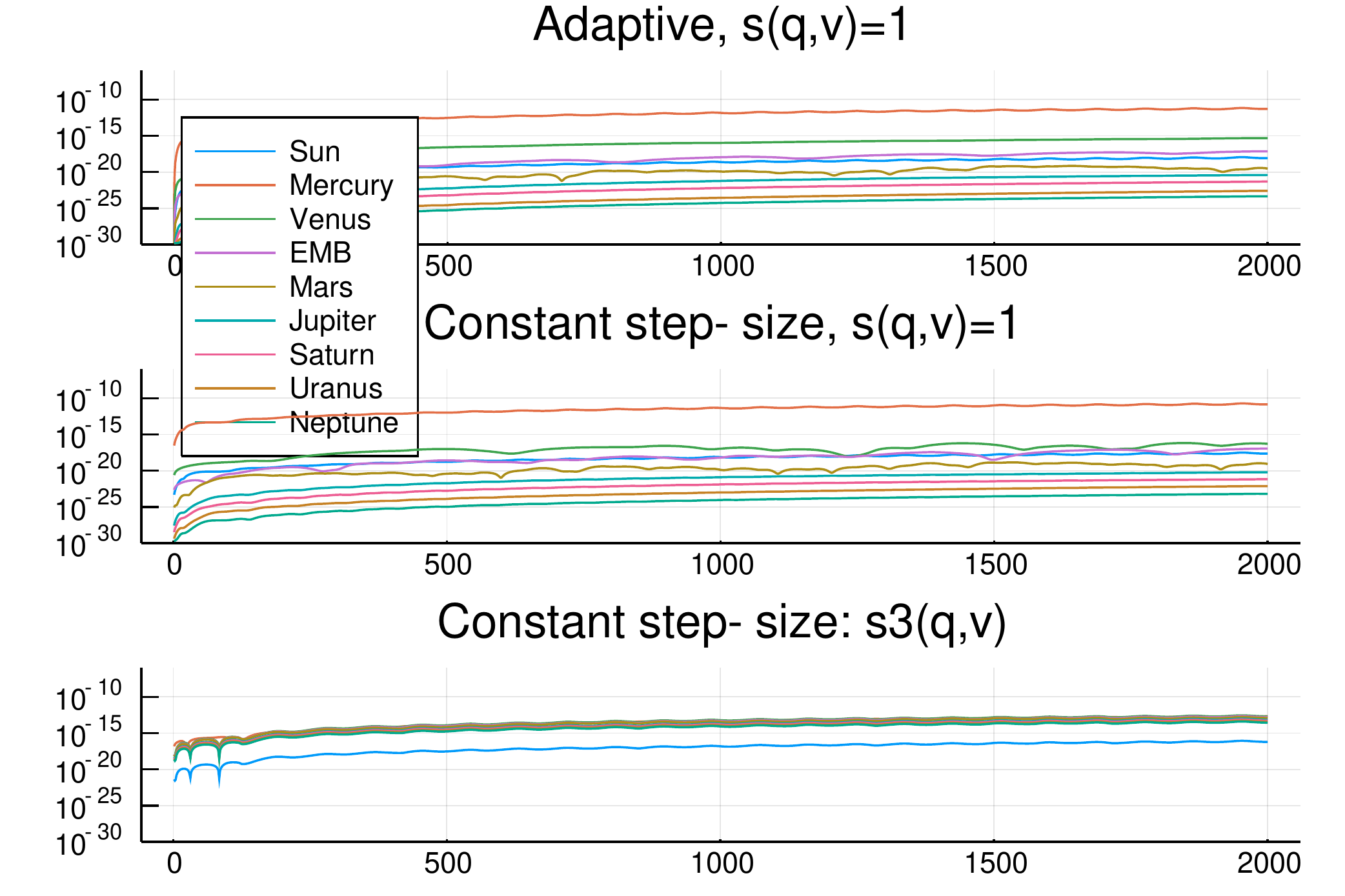}}
\end{center}
\caption{
\label{fig:QerrExample3}
Evolution of position errors  for the Solar System  9-body problem for 
adaptive numerical integration in physical time (top),  constant step-size in physical time (middle),  and constant step-size for the time-renormalized equations with $s_3(q,v)$ (bottom)}
\end{figure}

In Figure~\ref{fig:radiusExample3}, 
the  radius of convergence $\rho(q^k,v^k)$ is displayed together with the lower bound $1/L(q^k,v^k,\lambda_0)$  given by Theorem~\ref{th:L} scaled by a factor of $10$. The minimum and maximum of $\rho(q^k,v^k)L(q^k,v^k,\lambda_0)$ are $7.46384831$ and $14.48253$ respectively. One can observe that the variations of the radius of convergence is dominated by Mercury's influence, as expected. {In addition, one can observe smaller scale variations due to the influence of Venus. In fact, we have checked that the radius of convergence of the 9-body problem along that solution trajectory is virtually identical to the radius of convergence of the three-body problem consisting of the Sun, Mercury, and Venus. The local minima of the estimate of the radius of convergence in Theorem~\ref{th:L} approximately coincide with the local minima of the radius of convergence corresponding to the perihelion passages of Mercury, but the estimate fails detecting the smaller variations due to the influence of Venus.} 

The width of the strip around $[0,T]$ for the original equations, together with the widths corresponding to the time-renormalization functions $s_i(q,v)$, $i=1,2,3,4$ and their scaled versions are displayed in Table~\ref{tab:solar-system}. The time-renormalization functions $s_i(q,v)$, $i=1,2,3$, are not able to get wider scaled strips than in the case of the original equations. In this case, the widest strip is obtained in the case of $s_4(q)$.

\begin{table}[!ht]
\caption[Solar System example] 
{\small{Solar System example}}
\label{tab:solar-system}       
\centering
{%
\begin{tabular}{ l l l l l l } 
\\
                 & $s_0(q,v)=1$   & $s_1(q,v)$ & $s_2(q,v)$ & $s_3(q,v)$ &  $s_4(q)$ \\ 
 \hline 
\\
 Width of strip    & $41.02$      & $6.6117$    & $5.3245$    & $7.4277$  & $3.81365$  \\  \\
Scaled width    & $41.02$      & $40.07114$    & $38.007644$    & $28.8$  & $47.4334$  \\  \\   \hline
 \end{tabular}}
\end{table}

 We have included the integration with constant step-size of the original equations, in order to stress the fact that  neither adaptive step-size strategy nor time-renormalization are required for an accurate and efficient numerical integration of that example.
  The constant step-size integration of the original equations gives the largest maximum relative error in energy ($7.8976\times 10^{-16}$), followed by the adaptive integration of the original equations ($2.8864\times 10^{-16}$).  The maximum energy errors of the integration with  constant step-size with time-renormalization functions $s_i(q,v)$, $i=1,2,3,4$ are in all cases smaller. The smallest  maximum error in energy is obtained with   $s_3(q,v)$ ($7.12216\times 10^{-18}$).
As for the position errors $\|q_i(t^k)-q_i^k\|$ ($i=1,2,\ldots,9$), the time-renormalization function $s_3(q,v)$ gives the most precise results among the four time-renormalization functions $s_i(q,v)$, $i=1,2,3,4$.    In Figure~\ref{fig:QerrExample3}, the position errors for each of the 9 bodies are displayed for the adaptive numerical integration in physical time (top), the constant step-size integration in physical time (middle), and for the constant step-size implementation with time-renormalization function  $s_3(q,v)$. One can observe that the position error of Mercury is smaller in the later case, but this is achieved at the expense of having larger position errors for the rest of the bodies.

\subsection{Summary of numerical experiments}

{Although Theorem~\ref{th:L} does not give sharp estimates for the radius of convergence in the considered examples, the scaling between the actual radius and its estimate remains essentially  consistent along the solution of  any specific  problem, particularly in close encounters. We believe that the scaling factor does not increase significantly with the number of bodies provided that only a few of the terms $G\, (m_i + m_j)/\|q_i-q_j\|^2$ dominate in the expression for the estimate of the radius of convergence (as in single binary close encounters). In contrast, the scaling factor may increase (at most linearly) with the number of bodies whenever all such terms are of similar size.}

In the Pythagorean problem, a classical three-body problem with extreme close encounters, all the considered time-renormalization functions ($s_i(q,v)$, $i=1,2,3$, and $s_4(q)$) show similar behavior: Constant step-size numerical integration of the time-renormalized equations is  clearly more efficient than the adaptive integration of the original equations. There seem to be a correlation between the scaled width of the strip around the integration interval in $\tau$ for each  time-renormalization function, with the accuracy of the numerical integration with constant step-size.

The second example has been cooked up to illustrate the fact that $s_4(q)$ is not a uniform global time-renormalization function in the sense of Definition~\ref{def:timereg}. That is, there exists no $\beta>0$ such that all the solutions of  (\ref{eq:Nbodytau}) with  $s(q,v) := s_4(q)$ can be extended analytically in the strip $\{\tau \in \mathbb{C}\ : \ |\mathrm{Im}(\tau)| \leq \beta\}$. 
The function $s(q,v):= s_4(q)$ may be appropriate for the numerical integration with constant step-size of some solution trajectories of  (\ref{eq:Nbodytau}) {but not for trajectories involving  close encounters where the inertial term $\|v_i-v_j\|/(\|r_i-r_j\|)$ in the estimate $L(q,v,\lambda)^{-1}$ for the radius of convergence given in Theorem~\ref{th:L} dominates over the gravitational term (in planetary systems,  this is typically the case for planet-planet close encounters.)}

In the third example, the 9-body model of the Solar System, there is actually no need for time-renormalization, as it can safely be integrated numerically with constant step-size.  In fact, similar position errors have been obtained by integrating the original formulation of the considered initial value problem in constant step-size mode and with the adaptive implementation. The only close approaches correspond to the perihelion of Mercury, and are very mild compared to those occurring in previous two examples. This example has been chosen to illustrate the performance of time-renormalization in unfavorable conditions: 
The estimates of the radius of convergence given by Theorem~\ref{th:L} are able to detect the perihelion passages of Mercury, {but not the finer variations of the radius of convergence due to the influence of Venus.} Furthermore, the errors in energy of the numerical integrations of the time-renormalized equations are smaller than in the case of the integration with the physical time. The position errors of Mercury are also smaller for the time-renormalized formulations, but larger for the rest of the bodies. This is due to the hierarchical nature of the Solar System, and in particular of our 9-body model: the comparatively high-frequency oscillations due to Mercury's orbit have in practice limited  effect on the smoothness of the positions of the rest of the bodies in the original formulation. However, the high-frequency oscillations of Mercury's orbit are inherited by the time-renormalization function $s(q,v)$. This implies that the positions of all the bodies as functions of the  fictitious time $\tau$ become as oscillatory as Mercury's position.

\section{Concluding remarks}\label{sect:conclusions}
In this work, we have introduced new time-renormalizations defined as real analytic functions $s(q,v)$ depending on both velocities and positions. Our main contribution has been to exhibit  {\em uniform global} time-renormalizations such that any solution of the reformulated $N$-body system is defined (as a function of the complex variable $\tau$) on a strip of width $2\beta=0.0888886$ along the complete real line $\R$. 
As a by-product, a global power series representation of the solutions of the $N$-body problem is obtained.
Noteworthy, our results improve over previous ones~\cite{wang,babadzanjanz2} by being insensitive to the possible occurrence of vanishing mass-ratios in the system.  

{As a first  step in the derivation of  our uniform global time-renormalizations, we have obtained an estimate of the radius of convergence of the local power series expansion of the solution of the $N$-body problem. Although such estimates are generally not sharp (in our numerical experiments, the scaling factor between the actual radius of convergence and our estimates varies roughly between 3 and 15), they capture the main variations of the radius of convergence for trajectories with close encounters. }

Uniform global time-renormalizations allow for the numerical integration with constant time-steps (in the new fictitious time) without degrading the accuracy of the computed trajectories with close approaches. 
We have performed some preliminary numerical experiments on three different examples: two three-body problems with close encounters, and a 9-body problem corresponding to the Solar System. The performance of time-renormalization compared to a highly efficient adaptive general purpose integrator has been checked for short integration intervals. We plan to address more exhaustive numerical tests, and in particular, to check the long term behavior of the proposed time-renormalization functions for different symplectic numerical integrators. 

\section*{Acknowledgements}

AM and PC have received funding from the Ministerio de Econom\'{\i}a y Competitividad (Spain) through project MTM2016-77660-P (AEI/FEDER, UE).  PC acknowledges funding by INRIA through its Sabbatical program and thanks the University of the Basque Country for its hospitality. In addition, MA, JM, and AM have received funding from the Department of Education of the Basque Government through the Consolidated Research Group MATHMODE (IT1294-19)

%

\end{document}